\documentclass[12pt]{amsart}

\usepackage{amssymb}
\usepackage{graphicx,color}

\newtheorem{theorem}{Theorem}[section]
\newtheorem{proposition}[theorem]{Proposition}
\newtheorem{lemma}[theorem]{Lemma}

\theoremstyle{definition}

\newtheorem{problem}[theorem]{Problem}
\newtheorem{question}[theorem]{Question}

\newtheorem{remark}[theorem]{Remark}

\newcommand{\init}{\ensuremath{\mathrm{in}}\hspace{1pt}}

\begin{document}

\title[Lefschetz properties for products of linear forms]{Lefschetz properties for complete intersection ideals generated by products of linear forms}

\author[M. Juhnke-Kubitzke ]{Martina Juhnke-Kubitzke }
\address{
Universit\"{a}t Osnabr\"{u}ck,
Fakult\"{a}t f\"{u}r Mathematik,
Albrechtstra\ss e 28a,
49076 Osnabr\"{u}ck, GERMANY
}
\email{juhnke-kubitzke@uni-osnabrueck.de
}

\author[R.M. Mir\'o-Roig]{Rosa M. Mir\'o-Roig}
\address{Universitat de Barcelona,
Departament de Matem\`{a}tiques i Inform\`{a}tica, Gran Via de les Corts Catalanes
585, 08007 Barcelona, SPAIN
}
\email{miro@ub.edu}

\author[S. Murai]{Satoshi Murai}
\address{Department of Pure and Applied Mathematics,
Graduate School of Information Science and Technology
Osaka University,  Suita, Osaka, 565-0871, JAPAN
}
\email{s-murai AT ist.osaka-u.ac.jp}

\author[A. Wachi]{Akihito Wachi}
\address{Hokkaido University of Education
Department of Mathematics,
Kushiro, 085-8580 JAPAN
}
\email{wachi.akihito@k.hokkyodai.ac.jp}

\date{\today}

\thanks{
The first author was sponsored the German Research Council DFG GRK-1916,
 the second author was sponsored by   MTM2016-78623-P,
 the third author was sponsored by JSPS KAKENHI Grant 16K05102 and
 the fourth author was sponsored by JSPS KAKENHI Grant 15K04812.
}
\keywords{strong Lefschetz property, initial ideals, monomial ideals,  complete intersection, Hilbert function.}
\subjclass[2000]{13E10, 13C13, 13C40}

\begin{abstract}
In this paper, we study the strong Lefschetz property of artinian complete intersection ideals generated by products of linear forms.
We prove the strong Lefschetz property for a class of such ideals with binomial generators.
\end{abstract}

\maketitle

\tableofcontents

\section{Introduction}
Let $K$ be a field of characteristic zero and let $R=K[x_1,\dots,x_n]$ be the polynomial ring over $K$ in $n$ variables.
A graded artinian $K$-algebra $A:=R/I$ is said to have the {\em strong Lefschetz property} ({\em SLP} for short)
if there is a linear form $\ell \in A_1$ such that the multiplication
$$\times \ell ^s: A_{k-s} \to A_k$$ has maximal rank for all $k$ and all $s$, i.e., $\times \ell^s$
is either injective or surjective, for all $k$ and all $s$.
If the multiplication $\times \ell: A_{k-1}\to A_k$ has maximal rank for all $k$, then $A=R/I$ is said to have the {\em weak Lefschetz property} ({\em WLP} for short).
A linear form $\ell$, as above, is called a {\em strong Lefschetz element} (resp. {\em weak Lefschetz element}) of
$A$.
We also say that $I$ has the SLP (resp.\ WLP) if $R/I$ has the SLP (resp.\ WLP).
Though many algebras are expected to have the WLP or even the SLP,
establishing this property is often rather difficult and even in seemingly simple cases, such as complete intersections and ideals generated by products of linear forms, much remains unknown about the presence of the Lefschetz properties.

Lefschetz properties have been studied intensively and a large toolbox, containing different approaches and methods, to check if a graded artinian $K$-algebra $A$ has the WLP or the SLP has been developed. Their study is interesting not only because they put a lot of restrictions on the Hilbert function of a standard graded $K$-algebra but also since they have shown to be connected to a large number of problems, that appear to be unrelated at first glance. 
 Nevertheless, all  research results in this area are motivated and  owe their roots to the following theorem proved  by
Stanley in \cite{St}, Watanabe in \cite{W} and
Reid, Roberts and Roitman in \cite{RRR}: If $K$ is a field of characteristic zero, then the artinian monomial complete intersection ideal $I=(x_1^{a_1},\cdots ,x_n^{a_n})\subset R$
has the strong and, in particular, the weak Lefschetz property (see also \cite{HP} and \cite{I}).
 As a consequence we have that a
\emph{general}
complete intersection artinian ideal with
fixed generator degrees has both, the WLP and the SLP.
It is therefore natural to pose the following question:

\begin{question}[{\cite[Question 3.1]{MN}}]
\label{prob:AllCI}
Do \emph{all} artinian complete intersections have the WLP/SLP, in characteristic zero?
\end{question}

Some evidence that Question \ref{prob:AllCI} has a positive answer is given by the fact that \emph{all} artinian complete intersections in $3$ variables are known to have the WLP. However, it is a challenging and still open problem to decide whether all height 3 artinian complete intersections have the SLP. Similarly, the answer to Question \ref{prob:AllCI} for height 4 ideals is unknown. 

There are plenty of results concerning Lefschetz properties of ideals generated by powers of linear forms (see e.g.,\ \cite{AA,MMN-2012,M2016,SS}), and it is natural to consider generalizations of  those results to ideals generated by products of linear forms. 
From this point of view, it would be interesting to study Question \ref{prob:AllCI} for artinian complete intersections generated by products of linear forms. In this short note, we make a non-trivial contribution to this problem by providing a family of artinian complete intersection ideals of arbitrary height $n$, which have the SLP. More precisely, the following is our main result:

\begin{theorem}\label{mainresult}
Let $d_0,\dots,d_{n-1}$ be positive integers and let $a \in K$ with $a \ne 1$.
The algebra
$$R/(x_n^{d_0}(x_n-a x_1),x_{i}^{d_{i}}(x_{i}-x_{i+1})~|~ \ 1\le i \le n-1)$$
is an artinian complete intersection and has the SLP.
\end{theorem}

The structure of the paper is as follows. 
After recalling some basic results on Lefschetz properties in Section 2, we prove the main result of this paper (Theorem \ref{mainresult}) and an extension of it in Section 3. 
In Section 4, we present some research problems concerning Lefschetz properties of ideals generated by products of linear forms.

\vskip 4mm
\noindent {\bf Acknowledgements:} This work was started at the workshop ``Lefschetz Properties in Algebra, Geometry and Combinatorics,"
held at the Mittag-Leffler Institute (MLI) in July 2017.
The authors thank MLI for its kind hospitality.


\section{Background and preparatory results}
 In this section, for the sake of completeness, we recall the main tools and results that will be  used in the rest of the paper.

Throughout the following, we denote by $K$ a field
of characteristic zero and by $R =K[x_1,\dots,x_n]$ the graded homogeneous polynomial ring in $n$ variables over $K$.
%
%
Strong respectively weak Lefschetz elements of an artinian algebra $R/I$ are known to form a Zariski open, possibly empty,
subset of $(R/I)_1$. In other words, if the artinian $K$-algebra $R/I$ satisfies the strong or the weak Lefschetz property for some linear form, then it does so for a general linear form. However, for monomial ideals there is no need to consider a general linear form due to the following useful result:

\begin{proposition}[{\cite[Proposition 2.2]{MMN}}]
  \label{lem-L-element}
Let $I \subset R$ be an artinian monomial ideal. Then $R/I$ has the SLP if and only if  $x_1 + \cdots + x_n$ is a strong Lefschetz element for $R/I$.
\end{proposition}

We also recall the following well-known result, which can be seen as \emph{the} starting point of the study of Lefschetz properties (see \cite{RRR,St,W} and also \cite{HP,I}).

\begin{proposition}
  \label{monomialCI}
Let $d_1,\dots,d_n$ be non-negative integers and let $D=d_1+\cdots+d_n$.
Then $A=R/(x_1^{d_1+1},\dots,x_n^{d_n+1})=A_0\oplus \cdots \oplus A_D$ has the SLP. 
In particular, the multiplication map
$$\times (x_1+\cdots+x_n)^{D-2i}: A_i \to A_{D-2i}$$
is an isomorphism for $i<\frac D 2$.
\end{proposition}

The following result, that will be crucial for the proof of Theorem \ref{mainresult}, enables us to reduce the study of Lefschetz properties of an arbitrary artinian ideal $I$ to the one of a monomial ideal, by passing to an initial ideal of $I$.

\begin{proposition}[{\cite[Proposition 2.9]{We}}]
  \label{WLP_initialIdeal}
Let $I \subset R$ be an artinian ideal, $\tau$ a term order and  $\init_\tau(I)$ the initial ideal of $I$ with respect to $\tau$. If $R/\init_\tau(I)$ has the SLP (resp.\ WLP), then so has $R/I$.
\end{proposition}


\section{A class of complete intersections}

This section is dedicated to the proof of Theorem \ref{mainresult}. Thereby, we provide a class of artinian complete intersection ideals, generated by products of linear forms, that satisfy the SLP.

\begin{lemma}\label{initial}
Let $d_0,d_1,\dots,d_{n-1}$ be positive integers, $a \in K$ with $a \not \in \{0,1\}$, and let
$I=(x_n^{d_0}(x_1-ax_n),x_i^{d_i}(x_i-x_{i+1})~|~1\leq i\leq  n-1).
$
Then the initial ideal $\init_{\mathrm{lex}}(I)$ of $I$ with respect to the lexicographic order is equal to
$$(x_1^{d_1+1},\cdots,x_{n-1}^{d_{n-1}+1}, x_{i}x_{n}^{d_0+\cdots+d_{i-1}} ~|~1\leq i\leq  n).$$
In particular, $R/I$ is an artinian complete intersection.
\end{lemma}

\begin{proof}
Let
$$f_i = x_i^{d_i}(x_i- x_{i+1}) \quad \text{ for }1\leq i\leq n -1$$
and
$$ g_i = x_ix_n^{d_0+\cdots+d_{i-1}} -  \frac 1 a x^{d_0+\cdots+d_{i-1}+1}_n\quad  \text{ for } 1\leq  i\leq n.$$
We prove the first statement by showing that $G = \{ f_1, \dots , f_{n-1}, g_1, \dots , g_n \}$ is a Gr\"{o}bner basis  of $I$ with respect to the lexicographic order. 
In the following, we write $S(f, g)$ for the $S$-polynomial of $f$ and $g$. 
A routine computation implies the following properties of the possible $S$-polynomials: 

\begin{itemize}
\item $S(f_i, g_i) = -(x_{i+1}x_i^{d_i}x_n^{d_0+\cdots+d_{i-1}} - \frac 1 a x_i^{d_i}x_n^{d_0+ \cdots + d_{i-1}})$ reduces to $-\frac 1 {a^{d_i}} g_{i+1}$ with respect to $\{f_1,\dots,f_n,g_1,\dots,g_i\}$ for $1\leq i\leq  n - 1$.
\item $S(g_i, g_j) = - \frac 1 a (x_j x_n^{d_0+\cdots+ d_{j-1}+1} - x_ix_n^{d_0+\cdots+ d_{j-1}+1})$
reduces to zero with respect to $G$ for $1 \leq i < j\le n- 1$.
\item $\init_{\mathrm{lex}}(f_i)$ and $\init_{\mathrm{lex}}(f_j)$ are relatively prime if $1\leq i\ne j\leq n-1$.
\item $\init_{\mathrm{lex}}(f_i)$ and $\init_{\mathrm{lex}}(g_j)$ are relatively prime if $1\leq i\leq n-1$, $1\leq j\leq n$ and $i\neq j$.
\item $S(g_i,g_n)=-\frac 1 a x_n^{d_1+ \cdots + d_n+2}=- \frac{x_n}{a-1}g_n$ for $1\leq i\leq n-1$ .
\end{itemize}
The above facts combined with Buchberger's criterion guarantee that $G\subset  I$ and that $G$ is a Gr\"{o}bner basis of $I$ with respect to  the lexicographic order.

For the second statement, observe that, since $\init_{\mathrm{lex}}(I)$ contains pure powers of all  variables, $R/\init_{\mathrm{lex}}(I)$  and hence also $R/I$ is artinian. Moreover, as $I$ is generated by $n$ polynomials, $R/I$ is a complete intersection.
\end{proof}

Theorem \ref{mainresult} will finally follow from Proposition \ref{WLP_initialIdeal} and the next statement:

\begin{theorem}
\label{mainthmSLP}
Let $d_0,d_1,\dots,d_{n-1}$ be positive integers and let $$J=(x_1^{d_1+1},\cdots,x_{n-1}^{d_{n-1}+1}, x_{i}x_{n}^{d_0+\cdots+d_{i-1}} ~|~1\leq i\leq  n).$$
Then $R/J$ has the SLP.
\end{theorem}

\begin{proof}
Let $D=d_0+d_1+\cdots+d_{n-1}$.
It follows from Lemma \ref{initial}, that the Hilbert series of $J$ is the same as the one of a complete intersection ideal generated by polynomials of degrees $d_0+1,d_1+1,\dots,d_{n-1}+1$. Therefore, by Proposition  \ref{lem-L-element}, it suffices to prove that the multiplication
$$ \times (x_1+\cdots+x_n)^{D-2i}: (R/J)_i \to (R/J)_{D-i}
$$
is an isomorphism for $i<\frac D 2$.

Let $A=R/J$ and for $1\leq i\leq n-1$ let $M^{(i)}=(x_1,\dots,x_i)A$ be the ideal of $A$ generated by the variables $x_1,\dots,x_i$. 
Using induction on $k$, we first prove that the multiplication map
\begin{align}
\label{imaizumi}
\times (x_1+\cdots+x_n)^{D-2i}: M^{(k)}_i \to M^{(k)}_{D-i}
\end{align}
is an isomorphism for all $1\leq k\leq  n-1$ and all $i < \frac D 2$. 
For $k=1$, we have an isomorphism
$$R/(x_1^{d_1},x_2^{d_2+1},\dots,x_{n-1}^{d_{n-1}+1},x_n^{d_0}) (-1)= R/(J:x_1)(-1)
\stackrel{\times x_1} \longrightarrow M^{(1)}$$
and therefore, Proposition \ref{monomialCI} guarantees that the map \eqref{imaizumi} is an isomorphism in this case. \\
Suppose that $k>1$.
Since $M^{(k)}/M^{(k-1)}$ is an ideal of $A/M^{(k-1)}=R/(J+(x_1,\dots,x_{k-1}))$ generated by the  single element $x_k$, we have that
\begin{align*}
M^{(k)}/M^{(k-1)}  &\cong R/((J+(x_1,\dots,x_{k-1})):x_k)(-1)\\
&\cong K[x_k,x_{k+1},\dots,x_n]/(x_k^{d_k},x_{k+1}^{d_{k+1}+1},\dots,x_{n-1}^{d_{n-1}+1},x_n^{d_0+ \cdots +d_{k-1}})(-1),
\end{align*}
which, by Proposition \ref{monomialCI}, has the SLP. Using the induction hypothesis and the  following exact sequence 
\begin{align}
\label{exact}
0 \longrightarrow M^{(k-1)} \longrightarrow M^{(k)}
\longrightarrow M^{(k)}/M^{(k-1)} \longrightarrow 0,
\end{align}
we conclude that \eqref{imaizumi} is an isomorphism for all $1\leq k\leq n-1$ and all $i < \frac D 2$.

Finally, it follows from the short exact sequence
$$0 \longrightarrow M^{(n-1)} \longrightarrow A \longrightarrow A/M^{(n-1)} \cong K[x_n]/(x_n^{D+1}) \longrightarrow 0,$$
$A$ has the SLP.
\end{proof}

Theorem \ref{mainresult} is an almost immediate consequence of Lemma \ref{initial} and Theorem \ref{mainthmSLP}.

\begin{proof}[Proof of Theorem \ref{mainresult}]
Note that, by Proposition \ref{WLP_initialIdeal}, it is enough to show that $R/\init_{\mathrm{lex}}(I)$ has the SLP. 
If $a=0$, then $\init_{\mathrm{lex}}(I)=(x_n^{d_0+1},x_1^{d_1+1},\dots,x_{n-1}^{d_{n-1}+1})$ is an artinian complete intersection and it follows from Proposition \ref{monomialCI} that $R/\init_{\mathrm{lex}}(I)$ has the SLP.

If $a \ne 0$, it follows from Lemma \ref{initial} and Theorem \ref{mainthmSLP} that $R/\init_{\mathrm{lex}}(I)$ has the SLP.
\end{proof}

\begin{remark}
The sum of the variables $x_1+ \cdots +x_n$ may not be a Lefschetz element of $R/I$ in Theorem \ref{mainresult}.
Indeed, we have checked with Macaulay2 \cite{GS} that if $K=\mathbb R$, $n=7$, $d_0=d_1=\cdots =d_6=1$ and $a= \pm \frac 4 {\sqrt {3}}$, then $x_1+ \cdots +x_n$ is not a Lefschetz element for $R/I$. \end{remark}

In the remaining part of this section, we discuss an extension of Theorem \ref{mainresult}.
To do so, let us first fix some notation. Let $M_{n\times n}(K)$ denote the set of all $n\times n$ matrices with entries in $K$. 
To any matrix $A\in M_{n\times n}(K)$ and any tuple
$\vec{d} = (d_1, d_2, \ldots, d_n)$
of positive integers 
we associate the ideal $$I_{A,\vec{d}}:=\left(x_i^{d_i}\left(\sum _{j=1}^na_{ij}x_j\right) ~\mid~ 1\leq i\leq n\right)\subset R.$$
It is known that $I_{A,\vec{d}}$ is an artinian complete intersection ideal if and only if all principal minors of $A$ are non-zero
(see e.g., \cite[Lemma 2.1]{A}, where this equivalence is proved  when $\vec d=(1,1,\dots,1)$).
Since these ideals give a class of artinian complete intersection ideals, it is natural to ask if they have the SLP.
Using our main result Theorem \ref{mainresult} we are able to provide a positive answer to this problem in a special case:

\begin{theorem}
Let $A\in M_{n\times n}(K)$ be a matrix with non-zero principal minors. Assume that each row of $A$  has exactly two non-zero entries. 
Then $R/I_{A,\vec{d}}$ has the SLP.
\end{theorem}
The condition that in each row of the matrix $A$ are precisely two non-zero entries just says that the ideal $I_{A,\vec{d}}$ is generated by binomials.

\begin{proof}
We will show that~--~up to a change of coordinates~--~$R/I_{A,\vec{d}}$ is
isomorphic to an extension of a tensor product of algebras that have the SLP.
Since tensor products preserve the SLP  \cite[Theorem 3.34]{Book},
and the above extension also turns out to preserve the SLP, the claim follows.

First note that by permuting the variables $x_i$ and $x_j$, the ideal
$I_{A,\vec{d}}$ is changed to $I_{A',\vec{d'}}$, where $A'$ is obtained from $A$ by 
exchanging the $i$\textsuperscript{th} and $j$\textsuperscript{th} row as well as the $i$\textsuperscript{th} and $j$\textsuperscript{th} column.

Since all principal minors of $A$ are non-zero, we know that the diagonal entries $a_{ii}$ have to be non-zero. Moreover, as each row of $A$ contains exactly two non-zero entries, for any $1\leq i\leq n$ there exists a unique $j_i$ such that the entry in the $i$\textsuperscript{th} row and $j_i$\textsuperscript{th} column of $A$ is non-zero. We denote this entry by $b_{ij_i}$.

We first show the following claim.\\
{\sf Claim:} We can assume that $A$ is of the form \begin{equation}\label{TransformedA}
\begin{pmatrix}
B    &  *    &  *    &  \cdots  &  *       \\
{}   &  A_1  &  0    &  \cdots  &  0       \\
{}   &  {}   &  A_2  &  \cdots  &  0       \\
{}   &  {}   &  {}   &  \ddots  &  \vdots  \\
{}   &  {}   &  {}   &  {}      &  A_r
\end{pmatrix}.
\end{equation}
We associate a directed graph $G_A$ to the matrix $A$ in the following way:  
The set of vertices of $G_A$ is defined to be the set of 
symbols $a_{ii}$ and $b_{ij_i}$ ($1\leq i\leq n$).  
The set of edges is defined
to be the set of  all ordered tuples $(a_{ii}, b_{ij_i})$ and $(b_{ij_i}, a_{j_ij_i})$ ($1\leq i\leq n$). 
As $j_i$ is uniquely determined by $i$, we conclude that $a_{ii}$ and $b_{ij_i}$ both have outdegree  equal to $1$. 
%
Therefore, the finite graph $G_A$ has at least one (directed) cycle,
and distinct cycles are disjoint. If the vertex $b_{ij_i}$ belongs to a cycle, then so do $a_{ii}$ and $a_{j_ij_i}$ (as $b_{ij_i}$ only lies in the edges $(a_{ii},b_{ij_i})$ and $(b_{ij_i},a_{j_i}{j_i})$). Hence, the sets of row and column indices occurring in (vertices of) a cycle coincide.
It follows from the previous discussion that we can assume that $A$ is of the form  \eqref{TransformedA}, where $r\ge1$ is the number of cycles of $G_A$ and the $A_i$ are $n_i\times n_i$-matrix, whose row and column indices correspond to vertices forming a cycle. The matrix $B$ is an $n_0\times n_0$-matrix, whose row indices correspond to vertices $a_{ii}$ not contained in any cycle. This shows the above Claim.

Since the non-zero entries in the first $n_0$ rows and columns of $A$ correspond to vertices of $G_A$ not lying in a cycle, by permuting and scaling the variables we can assume that $B$ is upper triangular with diagonal entries equal to $1$. 
%

%
Furthermore, again by permutation of variables, we can suppose that the only non-zero off-diagonal entries of $A_k$ are the ones directly above the diagonal or the one in the bottom left corner of $A$. Note that, multiplying the $i$\textsuperscript{th} column of $A$ by $\lambda \in K\setminus\{0\}$ corresponds to a scaling of the variables ($x_i \mapsto \lambda x_i$), whereas multiplying the $i$\textsuperscript{th} row of $A$ by $\lambda$ does not change the ideal at all. Hence, by scaling the variables, we can even assume that $A_k$ has the following shape:
\begin{align*}
A_k =
\begin{pmatrix}
1       &  -1      &  0       &  \cdots  &  0       \\
0       &  1       &  -1      &  \ddots  &  \vdots  \\
\vdots  &  \ddots  &  \ddots  &  \ddots  &  0       \\
0       &  {}      &  \ddots  &  1       &  -1      \\
a       &  0       &  \cdots  &  0       &  1
\end{pmatrix},
\end{align*}
where $a\in K\setminus\{0\}$ depends on $A_k$.  
Note that $A_k$ gives the ideal in Theorem~\ref{mainresult}.
We finally infer from the previous argumentation that
\begin{align*}
R/I_{A,\vec{d}} \cong
\frac{
\left(
\bigotimes_{k=1}^r K[\text{$n_k$ variables}] / I_{A_k,\vec{d_k}}
\right)
[x_1,x_2,\ldots,x_{n_0}]
}{
(f_1,f_2,\ldots,f_{n_0})
},
\end{align*}
where $f_\ell$ ($1\leq \ell\leq n_0$) is a homogeneous binomial that is monic in $x_\ell$, 
and $\vec{d_k}$ ($1\leq k\leq r$) contains the entries of $\vec{d}$ corresponding to $A_k$.
The above tensor product has the SLP,
since each factor has the SLP by Theorem~\ref{mainresult}.
Then the right-hand side of the above equation can be 
considered as the repetition of a simple extension
$S[x_\ell]/(f_\ell)$ of an algebra $S$ that has the SLP
for $1\leq \ell\leq n_0$ (in reverse order). 
Finally the right-hand side has the SLP by \cite[Corollary~4.17]{Book},
which finishes the proof of the theorem.
\end{proof}

If $A$ is  an upper triangular matrix, then the initial ideal $\init_{\mathrm{lex}}(I_{A,\vec{d}})$ with respect to the lexicographic order is a monomial complete intersection. So, $R/\init_{\mathrm{lex}}(I_{A,\vec{d}})$ has the SLP and, applying Proposition \ref{WLP_initialIdeal}, we conclude that the same holds for $R/I_{A,\vec{d}}$.


\section{Final comments and open problems}
In this section, we will present some open questions for further research. 

To simplify the notation,
we will write $I_A=I_{A,(1,1,\dots,1)}$.
Our study of Theorem \ref{mainresult} and the ideals $I_{A,\vec d}$ is motivation by the following fact:
If $I$ a quadratic artinian complete intersection ideal generated by products of linear forms, then~--~by applying an appropriate change of coordinates~--~ it follows that $R/I$ is isomorphic to $R/I_{A,(1,\ldots,1)}$ for some $A$. 
%
%
%
%
Thus an affirmative answer to the next question will give an affirmative answer to Question \ref{prob:AllCI} for quadratic complete intersection ideals generated by products of linear forms.

\begin{problem} \label{prob:Minors}
Does  $R/I_A$ (or more generally $R/I_{A,\vec{d}}$) have the WLP/SLP, if all principal minors of $A$ are non-zero?
\end{problem}

We also propose some special instances of Problem \ref{prob:Minors}.
\begin{problem} Does $R/I_A$ (or more generally $R/I_{A,\vec{d}}$) have the WLP/SLP if $A$ is integral and all its principal minors are equal to  $\pm 1$?
\end{problem}

Since a positive definite symmetric matrix $A$ always has non-zero principal minors, we suggest to consider the following problem.

\begin{problem} Does $R/I_A$ (or more generally $R/I_{A,\vec{d}}$) have the WLP/SLP if $A$ is  a positive definite symmetric matrix?
\end{problem}

More generally, we suggest to study the following problem:

\begin{problem}
Let $I$ be an artinian complete intersection ideal generated by forms $f_i$ of  degree $a_i$. Assume that $f_i=\ell _1^{(i)}\cdots \ell _{a_i}^{(i)}$ is a product of $a_i$ linear forms. Does $R/I$ have the WLP/SLP?
More specifically, one can ask the same question for the case that all $f_i$ are of the same degree, i.e., $a_i=d$ for all $i$.
\end{problem}

As an extension of the previous problem, we propose the following problem:
\begin{problem} Study the WLP/SLP for ideals generated by products of linear forms.
\end{problem}

Concerning this last problem it is worthwhile to point out that there is a huge list of papers dealing with ideals generated by powers of linear forms. For more information on this subject the reader can see, for instance, \cite{HSS}, \cite{MMN-2012}, \cite{M2016} and \cite{SS}. The problem is really subtle since a minuscule change can alter the behavior of the WLP. Indeed, the ideals
$$I_1=(x_1^4,x_2^4,x_3^4,x_4^4,x_1x_2x_3x_4) \text{ and } I_2=(x_1^4,x_2^4,x_3^4,x_4^4,x_1x_2x_3(x_1+x_4))$$
have the same Hilbert function: $1 \ 4 \ 10 \ 20 \ 30 \ 36 \ 34 \ 24 \ 12 \ 4 \ 0 $
but $I_1$ never has the WLP while $I_2$ does have the WLP.

\begin{remark} One might suspect that Problem \ref{prob:Minors} generalizes to almost complete intersection ideals generated by quadrics, which are products of linear forms. However, in \cite[Theorem 2.12]{M2016}, the second author proved that this is not the case. Indeed, for $n=6$ and all $n\ge 8$, the artinian ideal $I=(\ell^2_1,\ldots, \ell^2_{n+1})\subset  K[x_1, \cdots,x_n]$ generated by the square of $n+1$ general linear forms fails the WLP. Nevertheless, we have checked with Macaulay2 that the artinian  ideal $I=(\ell_1\ell _1',\ldots, \ell_{7}\ell _{7}')\subset  K[x_1, \cdots,x_6]$ generated by the products of  general linear forms  has the WLP.
\end{remark}

\end{document}